\documentclass[preprint,12pt]{amsart}

\usepackage[margin=1.5in]{geometry}

\usepackage[utf8]{inputenc}
\usepackage[T1]{fontenc}
\usepackage{lmodern}
\usepackage[english]{babel}
\usepackage{amsfonts,amsmath,mathrsfs,amssymb,amsthm,amscd,latexsym,amstext,amsxtra}
\usepackage{graphicx}
\usepackage[all]{xy}
\usepackage{enumerate}
\usepackage{booktabs}
\usepackage{caption}

\usepackage[style=numeric, backend=biber]{biblatex}

\usepackage{tikz}
\usetikzlibrary{cd}

\theoremstyle{plain}

\newtheorem{theorem}{Theorem}[section]
\newtheorem{corollary}[theorem]{Corollary}
\newtheorem{lemma}[theorem]{Lemma}

\theoremstyle{definition}
\newtheorem{definition}[theorem]{Definition}

\newtheorem{remark}[theorem]{Remark}

\renewcommand{\P}{\mathbb{P}}

\newcommand{\Z}{\mathbb{Z}}

\newcommand{\OO}{\mathcal{O}}

\newcommand{\Div}{\mbox{Div}}

\renewcommand{\div}{\mbox{div}}

\newcommand{\set}[1]{\left\lbrace #1 \right\rbrace}

\makeatletter
\def\ps@pprintTitle{%
  \let\@oddhead\@empty{}
  \let\@evenhead\@empty{}
  \let\@oddfoot\@empty{}
  \let\@evenfoot\@oddfoot{}
}
\makeatother

\title{Secants to the Kummer Variety and the minimal Cohomological Class}
\author{José Alejandro Aburto}
\thanks{Supported by ANID, Doctorado Nacional 2018, folio 21180924}
\email{jose.aburto@ug.uchile.cl}


\begin{document}

\maketitle

\begin{abstract}
We prove that, under certain conditions, the existence of a curve of $(m+2)$-secants to the Kummer variety of an indecomposable principally polarized abelian variety $X$,
represents $m$-times the minimal cohomological class in $X$. 
In the case of $m=2$, we find an involution of such curve which proves that \(X\) is a Prym variety by results of Krichever-Grushevsky.
This continues the work of Beauville and Debarre, who asked about the relation between some geometric properties of abelian varieties in a way to obtain a stratification of its moduli space.
\end{abstract}

\section{Introduction}
The origin of the problem of secants to the Kummer variety comes from the Schottky problem; 
that is, to determine which complex principally polarized abelian varieties arise as Jacobian varieties of complex curves.
It is a known fact (see \cite{Gunning1982feb}) that a Jacobian admits a \(4\)-dimensional family of trisecant lines to its Kummer variety.
Welters conjectured in \cite{Welters1984} that the existence of one trisecant line is enough to give such a characterization.
The conjecture resulted to be true. It was proved by Krichever in \cite{Krichever2010} and \cite{KricheverFlexSolved}  using analytic methods,
and an algebro-geometric proof of this theorem was published by Arbarello-Codogni-Pareschi \cite{Arbarello2021} for the degenerate case and the flex case.
This gives a very geometric solution to the Schottky problem.


Debarre proved in \cite{VersDebarre1992} that if the abelian variety $X$ contains a curve that parametrices $(m+2)$-secants, this implies, under certain geometric conditions, that the curve is $m$-times the minimal cohomological class in $X$. But, this is done under certain geometrical conditions.

The aim of this article is to remove several conditions required in previous results, such as $\mbox{End}(X)\simeq\mathbb{Z}$.
Then, we review the quadrisecant case ($m=2$), in which a certain involution of the curve in study appears, the existence of this involution implies that \(X\) is a Prym variety by a result of Krichever-Grushevsky in \cite{GrushevskyKrichever2010PairOf4secants}.
We finish by proving that the existence of a curve of $(m+2)$-secant ($m$-planes) to the Kummer variety can be achieved by the existence of a finite number of $(m+2)$-secants to the Kummer variety of $X$,
 more specifically, by one degenerate \((m+2)\)-secant plus \((m-1)\) honest (non-degenerate) secants.


\section{Preliminaries}
Let $(X,\lambda)$ be an indecomposable principally polarized abelian variety of dimension $g$, and let $\Theta$ be a symmetric representative of the polarization $\lambda$.
We can take a basis $\set{\theta}$ of $H^0(X,\OO_X(\Theta))$ and a basis $\set{\theta_0,\ldots,\theta_{N}}$ (with $N=2^g-1$) of the $2^g$-dimensional space $H^0(X,\OO_X(2\Theta))$ which satisfies the addition formula: 
	\begin{equation*}
		\theta(z+w)\theta(z-w) = \sum_{j=0}^N \theta_j(z)\theta_j(w),
	\end{equation*}
for all $z,w\in X$. 

Since symmetric representatives $\Theta$ of the polarization $\lambda$ differ by translations by points of order 2,
 we have that the linear system $\left|2\Theta\right|$ is independent of the choice of $\Theta$.
 It defines a morphism $K:X\to\left|2\Theta\right|$ called the \emph{Kummer morphism},
 whose image is the \emph{Kummer variety} $K(X)$ of $X$, and is isomorphic to $X/\pm 1$. 
Note that we may identify $K$ with ${[\theta_0:\cdots:\theta_{N}]}$.

Let $x\in X$ be any point, we write $\Theta_x$ for the divisor $\Theta+x$ and $\theta_x$ for the section $z\mapsto\theta(z-x)$ of $H^0\left(X,\OO_X(\Theta_x)\right)$.

\begin{definition}
	Let $Y\subset X$ be an Artinian scheme of length $m+2$. Define the scheme:
		\begin{equation*}
			V_Y	=	2\set{\zeta\in X\;:\; \exists W\in\mathbb{G}\left(m, 2^g-1\right)\text{ such that } \zeta+Y\subset K^{-1}W}.
		\end{equation*}
	We will say that $X$ satisfies the $(m+2)$-secant condition if such a subscheme exists and $K$ restricted to $\zeta+Y$ is an embedding.
\end{definition}

\section{A curve of $(m+2)$-secants and the minimal class}

Before giving the main result of this section, we prove a little lemma:
\begin{lemma}\label{lemma:tecLemma3components}
	Let $(X,\lambda)$ be an indecomposable principally polarized abelian variety of dimension $g>m$, let $\Theta$ be a symmetric representative of the polarization $\lambda$.
	Suppose that $Y=\left\lbrace a_1,\ldots,a_{m+2}\right\rbrace$ is a reduced subscheme of $X$ such that $V_Y$ contains an irreducible complete curve $\Gamma$ that generates $X$ and assume that
	${\Theta_{a_2}\cap \cdots\cap \Theta_{a_{m+2}}}$ is a complete intersection. If 
		\begin{displaymath}
			\Theta_{a_3}\cap\cdots\cap\Theta_{a_{m+2}}\subset \Theta_{a_1}\cup\Theta_{a_2}\cup Z
		\end{displaymath}
	for a subscheme $Z$ of $X$, then 
		\begin{displaymath}
			\Theta_{a_3}\cap\cdots\cap\Theta_{a_{m+2}}\subset Z.
		\end{displaymath}
\end{lemma}
\begin{proof}
	Let $W\subset \Theta_{a_3}\cap\cdots\cap\Theta_{a_{m+2}}$ be an irreducible component and suppose that $W$ is contained in either $\Theta_{a_1}$ or $\Theta_{a_2}$.
If $W$ is contained in $\Theta_{a_2}$ we contradict that ${ \Theta_{a_2}\cap\cdots\cap\Theta_{a_{m+2}} }$ is a complete intersection. 
On the other hand, if $W$ is contained in $\Theta_{a_1}$, this implies that $\zeta+W$ is contained in ${\Theta_{a_1+\zeta}\cap \Theta_{a_3+\zeta}\cap\cdots \Theta_{a_{m+2}+\zeta}}$, but
	\[{ \Theta_{a_1+\zeta}\cap \Theta_{a_3+\zeta}\cap\cdots \Theta_{a_{m+2}+\zeta}\subset\Theta_{a_2+\zeta}\cup \Theta_{-a_2-\zeta}, }\]

	so that $\zeta+W\subset \Theta_{-a_2-\zeta}$ (in the other case $W\subset \Theta_{a_2}$). That is, \(W\subset \Theta_{-a_2-2\zeta}\) for any \(\zeta\in \Gamma\), therefore $W=\emptyset$.
	This finishes the proof.
\end{proof}

The purpose of this section is to prove the following result:
\begin{theorem}\label{theorem:MainCurveOfmSecants}
		Let $(X,\lambda)$ be an indecomposable principally polarized abelian variety of dimension $g$, let $\Theta$ be a symmetric representative of the polarization $\lambda$.
		Suppose that $X$ has the \((m+2)\)-secant property, with a minimal \(m\in\mathbb{Z}_{>0}\), \(g<m\).
		Let $Y=\set{a_1,\ldots,a_{m+2}}$ be a reduced subscheme of $X$ and suppose that it satisfies the following conditions:
		\begin{enumerate}
			\item $V_Y$ contains an irreducible complete curve $\Gamma$ that generates $X$.
			\item The scheme $\Theta_{a_2}\cap\cdots\cap\Theta_{a_{m+2}}$ is a complete intersection. 
			\item The points $\left(a_1-a_j\right)_{2\leq j\leq m+2}$ are $\Z$-linearly independent.
		\end{enumerate}

		Then, $\Gamma$ is $m$ times the minimal cohomology class of $\Theta$.
\end{theorem}
\begin{proof}
	Since we have that the points $K\left(\zeta+a_1\right),\ldots,K\left(\zeta+a_{m+2}\right)$ lie on an $m$-plane, we have that there exist constants $\alpha_j$ such that
	\[\sum_{i=1}^{m+2} \alpha_i \theta_j(\zeta+a_i)=0\]
Multiplying by $\theta_j(z+\zeta)$ and taking the sum over $j$ we get
	\[\sum_{i=1}^{m+2} \alpha_i \theta(z+2\zeta+a_i)\theta(z-a_i)=0.\]
In particular, restricting to $\Theta_{a_3}\cap\cdots\cap\Theta_{a_{m+2}}$ we obtain
\begin{align*}
	\alpha_1 \theta_{-2\zeta-a_1}\theta_{a_1}+\alpha_2\theta_{-2\zeta-a_2}\theta_{a_2} & = 0,
	\intertext{and since this is true for all $\zeta\in\frac12 \Gamma$, it is true for $\zeta'\in\frac12 \Gamma$:}
	\alpha_1' \theta_{-2\zeta'-a_1}\theta_{a_1}+\alpha_2'\theta_{-2\zeta'-a_2}\theta_{a_2} & = 0.
\end{align*}
Multiplying the first equation by $\alpha_2'\theta_{-2\zeta'-a_2}\theta_{a_2}$ and the second one by $\alpha_2\theta_{-2\zeta-a_2}\theta_{a_2}$ and substracting we get
	$$\alpha_1\alpha_2' \theta_{-2\zeta-a_1}\theta_{-2\zeta'-a_2}\theta_{a_1}\theta_{a_2}-\alpha_1'\alpha_2 \theta_{-2\zeta'-a_1}\theta_{-2\zeta-a_2}\theta_{a_1}\theta_{a_2}=0.$$
Factor by $\theta_{a_1}\theta_{a_2}$, denote the remaining term as $A$. We get that
	$$\Theta_{a_3}\cap\cdots\cap\Theta_{a_{m+2}}\subset \set{A=0}\cup\Theta_{a_1}\cup\Theta_{a_2}.$$
By Lemma \ref{lemma:tecLemma3components} we get that
	$$\Theta_{a_3}\cap\cdots\cap\Theta_{a_{m+2}}\subset \set{A=0},$$
and so, 
	$$\Theta_{a_3}\cap\cdots\cap\Theta_{a_{m+2}}\cap\Theta_{-2\zeta-a_2}\subset \Theta_{-2\zeta-a_1}\cup\Theta_{-2\zeta'-a_2}.$$
Equivalently, this means that the points
	$$K(\zeta+b_1),\cdots,K(\zeta+b_{m+2})$$
lie on an $m$-plane, where
	\begin{align*}
		2b_1	& = 2\zeta'+2a_3+a_1+a_2 	\\
		b_j		& = b_1-a_3+a_{j+2}			\\
		b_{m+1}		& = a_2+a_{3}-b_1		 \\
		b_{m+2}		& = a_1+a_{3}-b_1
	\end{align*}
	for $2\leq j\leq m$.

In particular, $X$ admits a two dimensional family of $(m+2)$-secants, since now the $b_i$ depend on a one dimensional parameter $\zeta'$.

Let \(\rho\) be the endomorphism of \(X\) associated to \(\Gamma\) and \(\Theta\). We will prove under the above hypotheses that $\rho=m\cdot\mbox{Id}_X$, and so $\Gamma$ is $m$ times the minimal class.
In what comes next we will follow \cite[Section 4]{VersDebarre1992}.
First we will show that \(\rho=M\cdot\mbox{Id}_X\) for some \(M\in\Z\) and then we compute that \(M=m\).
%
%
%
\\
Let $N$ be the normalization of $\Gamma$, $\tilde{\Gamma}\subset 2^{-1}\left(\Gamma\right)$ an irreducible complete curve and $\tilde{N}$ its normalization.
We have the following commutative diagram
\begin{center}
	\begin{tikzcd}
		\tilde{N} \arrow[r, "\psi"]\arrow[d, "\tilde{\pi}"'] & N \arrow[r]\arrow[d,"\pi"'] & J(N)\arrow[d,"\gamma"] \\
		\tilde{\Gamma}\arrow[r, "2"'] & \Gamma\arrow[r] & X
	\end{tikzcd}
\end{center}
where $\gamma$ is the homomorphism induced by $\pi$. We observe that $\psi$ is a Galois morphism with Galois group isomorphic to those $2$-torsion points of $X$ that leave $\tilde{\Gamma}$ invariant.

By \cite[Lemma 3.7]{VersDebarre1992} there exist sections
\begin{equation*}
	\lambda_j\in H^0\left(\tilde{N},\tilde{\pi}^* \OO_X\left(2\Theta_{-b_1}+\cdots+2\Theta_{-b_{m+2}}-2\Theta_{-b_j}\right)\right)
\end{equation*}
such that for all \(j=1,\ldots,2^g\), \(\sum_{i=1}^{m+2} \lambda_i\pi^* t^*_{b_i}\theta_j=0\) on \(\tilde{N}\). Now each section \(\lambda_i\) is stable by the Galois group of $\psi$ and so $\div(\lambda_i)=\psi^* E_i$ for some divisor $E_i$ on $N$. 
Note that by the Theorem of the Square, \(\OO_\Gamma\left(\Theta_{b_i-b_j}\right)\simeq \OO_\Gamma\left(E_i-E_j\right)\). Hence, if \(\mathcal{S}:\Div(N)\to X\) is the pushforward by $\pi$
composed with the addition map of \(0\)-cycles, we have that
	\[\rho_\Gamma\left(b_i-b_j\right)=\mathcal{S}\left(E_i-E_j\right).\]
Moreover, $\lambda_i$ vanishes at a point $\zeta$ only if the $m+1$ points $K(b_p)$ for $p\neq i$ belong to a \((m-1)\)-plane,
unless they coincide, \emph{i.e.} $\zeta+b_p=-\zeta-b_q$ for $j\not\in\set{p,q}$, that is, the image in $X$ is $2\zeta=-b_p-b_q$.
Hence

%
%

\[\mathcal{S}(E_i-E_j)=\sum_{\substack{p,q\neq i \\ p<q }} \delta_{pq}\left(\zeta'\right)(-b_p-b_q)-\sum_{\substack{p,q\neq j \\ p<q }} \delta_{pq}\left(\zeta'\right)(-b_p-b_q)\]
for certain integers $\delta_{pq}\left(\zeta'\right)\geq 0$. Since the points $b_p-b_q$ for $1\leq p<q\leq m$ and $-b_{m+1}-b_{m+2}$ depend on $\zeta'$, we can choose $\zeta'$ so that
$b_p-b_q$ and $-b_{m+1}-b_{m+2}$ are either not on $\Gamma$ or are smooth points of $\Gamma$. This way, we get that $\delta_{pq}\in\set{0,1}$.

Note that $-b_p-b_q$ is constant if $p<m+1$ and $q\geq m+1$. Therefore,
\begin{equation*}
	\rho\left(2\zeta'\right)=2M\zeta' +\text{constant}
\end{equation*}
where
$M=\left(\sum_{\substack{1\leq p<q\leq m \\ p\text{ or }q \text{ equals }j}}\delta_{pq}+\sum_{\substack{1\leq p<q\leq m \\ p\text{ or }q \text{ equals }i}}\delta_{pq}\right)$ that is constant.

In particular, for a general $z\in\Gamma$ we have that $\rho(z)=Mz$.
Since \(\Gamma\) generates \(X\) we get \(\rho=M\cdot\mbox{Id}_X\) with \(M\in\mathbb{Z}\). 
Now, we show that $M=m$ using an argument due to Debarre \cite{VersDebarre1992}.

Note that the fact that \(\OO_\Gamma\left(\Theta_{b_i-b_j}\right)\simeq \OO_\Gamma\left(E_i-E_j\right)\) is of degree \(0\) immediately implies that the cardinality, namely \(N\), of the following set is independent of \(j\):
	\begin{equation*}
		\set{p\in\left\lbrace 1,\ldots,m+2\right\rbrace\;:\; p\neq j \text{ and }(-a_j-a_p)\in \Gamma}.
	\end{equation*}

Analogous to the previous computations, since $K\left(\zeta+a_1\right),\ldots,K\left(\zeta+a_{m+2}\right)$ lie on an $m$-plane, we have that
\begin{align*}
	\rho(a_i-a_j)	& = \sum_{\substack{p,q\neq i\\ p<q\\ -a_p-a_q\in\Gamma}} (-a_p-a_q) - \sum_{\substack{p,q\neq j\\ p<q\\ -a_p-a_q\in\Gamma}} (-a_p-a_q) \\
					& = -\sum_{\substack{p\neq i,j \\ -a_j-a_p\in\Gamma}} (a_j+a_p) + \sum_{\substack{q\neq i,j \\ -a_i-a_p\in\Gamma}} (a_i+a_q).
\end{align*}
Since \(\rho=M\cdot\mbox{Id}_X\), with \(M\in\Z\), then
\begin{equation*}
	M(a_i-a_j) = 	\begin{cases}
						N(a_i-a_j) + \sum_{\substack{ p\neq i,j \\ -a_i-a_p}} a_p-\sum_{\substack{ q\neq i,j \\ -a_j-a_q}} a_q & \text{if } -a_i-a_j\not\in\Gamma \\
						(N-1)(a_i-a_j) + \sum_{\substack{ p\neq i,j \\ -a_i-a_p}} a_p-\sum_{\substack{ q\neq i,j \\ -a_j-a_q}} a_q & \text{if } -a_i-a_j\in\Gamma.
					\end{cases}
\end{equation*}
Since \((a_1-a_j)_{2\leq j\leq m+2}\) are $\Z$-linearly independent, this leads to two cases:
\begin{itemize}
	\item Let $M=N$ and no $(a_i-a_j)$ lie on $\Gamma$. This implies that $N$ is null, and so is $\rho$ which is absurd.
	\item Let $M=N-1$ and $(a_i-a_j)\in\Gamma$ for all $i,j$. Then, $N=m+1$ and $\rho=m\cdot\mbox{Id}_X$. 
\end{itemize}
Therefore, $\Gamma$ is $m$ times the minimal class.
\end{proof}

\begin{remark}
	\begin{itemize}
		\item[]
		\item This theorem improves \cite[Theorem 4.1]{VersDebarre1992}, since we remove the condition over the endomorphism ring to be $\mathbb{Z}$.
		\item The condition of $\Theta_{a_2}\cap\cdots\cap\Theta_{a_{m+2}}$ is not easy to remove, note that in the case of $m=2$;
		as in \cite[p.617]{BeauvilleDebarreSurLeproblemedeSchottkyPrym}; we could obtain an abelian variety which satisfies the condition of quadrisecant planes but it is not necessary a Prym variety, so it is not clear if one gets a curve that is twice the minimal class.
	\end{itemize}
\end{remark}



\section{A curve of quadrisecant planes}
It is worth looking at the case $m=2$ since we can remove an additional hypothesis, that is,
we no longer need for $(a_1-a_j)_j$ to be $\mathbb{Z}$-linearly independent.

\begin{theorem}\label{theorem:curve4Secant}
    Let $(X,\lambda)$ be an indecomposable principally polarized abelian variety of dimension $g>m$,
    let $\Theta$ be a symmetric representative of the polarization $\lambda$.
	Suppose that $X$ does not possess the trisecant property.
    Let $Y=\set{a_1,a_2,a_3,a_{4}}$ be a reduced subscheme of $X$ and suppose that it satisfies the following conditions:
		\begin{enumerate}
			\item \(V_Y\) contains an irreducible complete curve $\Gamma$ that generates $X$.
			\item The scheme $\Theta_{a_2}\cap\Theta_{a_3}\cap\Theta_{a_4}$ is a complete intersection.
		\end{enumerate}
	Then, $\Gamma$ is twice the minimal cohomology class of $\Theta$.
\end{theorem}
\begin{proof}
    As in the proof of Theorem \ref{theorem:MainCurveOfmSecants}, we get that the endomorphism ${\rho=(\delta_{12}+\delta_{34})\cdot\mbox{Id}_X}$ where $\delta_{12},\delta_{34}\in\set{0,1}$. So we have three cases:
    \begin{itemize}
        \item $\rho=0$. This case is discarded immediately.
        \item $\rho=\mbox{Id}_X$. Then $\Gamma$ is the minimal class. By the Matsusaka-Ran Criterion (see \cite[pp.~341]{BirkenhakeLange2004}) this implies that $X$ is a Jacobian variety, which contradicts our hypothesis.
        \item $\rho=2\cdot\mbox{Id}_X$. Then $\Gamma$ is twice the minimal class. This proves the theorem.
    \end{itemize}
\end{proof}

\begin{remark}
	This is an improvement of Debarre's Theorem 5.1 in \cite{VersDebarre1992} and a variant of Theorem \ref{theorem:MainCurveOfmSecants}. If $a_1-a_2,a_1-a_3,a_1-a_4$ generate $X$, using Debarre's Theorem 5.1, we get that $\Gamma$ generates $X$, and so we require a weaker condition.
\end{remark}

\begin{corollary}\label{corollary:involution4secants}
	With the hypotheses of Theorem \ref{theorem:curve4Secant}, let $s=a_1+a_2+a_3+a_4$ be the sum of the four reduced points.
	Then,  the involution $z\mapsto -z-s$ on $X$ restricts to an involution of $\Gamma$.
\end{corollary}
\begin{proof}
	The previous proof shows that $\delta_{12}=1$ for a generic $\zeta'\in\frac12 \Gamma$. 
	Therefore, this means that for a generic $\zeta'\in\frac12 \Gamma$, $-b_1-b_2\in\Gamma$. However,
		$$-b_1-b_2=-2\zeta'-a_1-a_2-a_3-a_4.$$
\end{proof}

\begin{remark}
	We note that if we take $x\in\frac12 s$, then the involution $z\mapsto -z$ restricts to the curve $C:=\Gamma+x$. Here we have that $C\subset V_{t^* Y}$, so if $\eta\in C$ and we write
	$q_i:=a_i-x_i$, then
	\begin{align*}
		\Pi_\eta 	& =\mbox{span}\set{K\left(\eta +q_1\right),K\left(\eta+q_2\right),K\left(\eta +q_3\right),K\left(\eta +q_4\right)} \\
		\Pi_{-\eta}	& = \mbox{span}\set{K\left(\eta -q_1\right),K\left(\eta -q_2\right),K\left(\eta -q_3\right),K\left(\eta -q_4\right)}
	\end{align*}
	are planes in $\P^{2^g-1}$. This pair, by \cite{MR2656091} is necessary and sufficient for $\Gamma$ to be an Abel-Prym curve.
	Although, they use complicated analytic methods, there is work in progress to give a fully algebro-geometric proof.
\end{remark}

This result generalizes the main theorem in \cite{DebarreTrisecants1},
which was improved several years later in \cite{Arbarello2021}.
We believe that the main result in \cite{Arbarello2021} on the reducedness of a certain scheme \(\Sigma\)
can be used to remove the hypothesis that \(\Theta_{u}\cdot\Theta_{-u}\cdot\Theta_{b_1}\ldots\cdot\Theta_{b_{m-1}}\)
is reduced. Sadly, we has been unable to achieve this so far.

\printbibliography{}
\end{document}